\documentclass[a4paper,leqno,12pt]{amsart}

\usepackage{epsf,graphicx,epsfig}
\usepackage{amscd}
\usepackage{amsmath,latexsym,amssymb,amsthm}
\usepackage[nospace,noadjust]{cite}
\usepackage{textcomp}
\usepackage{setspace,cite}
\usepackage{lscape,fancyhdr,fancybox}
\usepackage[all,cmtip]{xy}
\usepackage{tikz}
\usetikzlibrary{shapes,arrows,decorations.markings}
\usepackage{graphicx}

\usepackage{color}
\usepackage{url}
\usepackage{enumerate}
\usepackage[mathscr]{euscript}

  \theoremstyle{plain}

\swapnumbers
    \newtheorem{thm}{Theorem}[section]
    \newtheorem{prop}[thm]{Proposition}
   \newtheorem{lemma}[thm]{Lemma}
    \newtheorem{corollary}[thm]{Corollary}
    
    \newtheorem{subsec}[thm]{}
\theoremstyle{definition}
    \newtheorem{defn}[thm]{Definition}
    \newtheorem{exam}[thm]{Example}

\theoremstyle{remark}
     \newtheorem{remark}[thm]{Remark}

\title{}
\author{}
\date{}
\usepackage{amssymb}

\usepackage{hyperref}
\hypersetup{
	colorlinks,
	citecolor=blue,
	filecolor=black,
	linkcolor=blue,
	urlcolor=black
}

\begin{document}
\title{Equivariant one-parameter deformations of associative algebras}

\author{Goutam Mukherjee}
\email{gmukherjee.isi@gmail.com}
\address{Stat-Math Unit,
Indian Statistical Institute, Kolkata 700108,
West Bengal, India.}

\author{Raj Bhawan Yadav}
\email{rbyadav15@gmail.com}
\address{Indian Statistical Institute, North-East Center, (temporary campus)
Tezpur University, Assam 784028}

\subjclass[2010]{Primary: 13D03, 13D10, 14D15, 16E40, Secondary: 55M35, 55N91.}
\keywords{Group actions, Hochschild cohomology, formal deformations, equivariant cohomology}

\thispagestyle{empty}

\begin{abstract}
We introduce an equivariant version of Hochschild cohomology as the deformation cohomology to study equivariant deformations of associative algebras equipped with finite group actions.  \end{abstract}
\maketitle


\vspace{0.5cm}

\section{Introduction}
The idea of algebraic deformation theory was introduced by M. Gerstenhaber \cite{G1, G2, G3, G4, G5}. He introduced deformation theory for associative algebras. His theory was extended to Lie algebras by A. Nijenhuis and R. Richardson \cite{NR1}, \cite{NR2}, \cite{NR3}. Let $k$ be a fixed field. Given a type of algebra $\mathcal{A}$ over $k,$ (usually denoted by {\bf As, Lie, Com, Dias, Leib}, etc.), its formal one-parameter deformation theory has been extensively studied in the literature, following Gerstenhaber' work. In recent time, more general algebraic deformation theory has been studied for Loday algebras \cite{L1, L2, L3} and other type of algebras by several authors (see \cite{FiFu, FMM, Fox, MM1, MM2, AM1, AG, GS1, GS2, DY1, DY2}). 

Let $X$ be an object in a category of algebras. Roughly speaking, a deformation of $X$ is a family ${X_t}$ of objects whose structures are obtained by ``deforming'' the structure on $X$ as $t$ varies over a suitable space of parameters in a smooth way.  If $A$ is an algebra over a commutative ring $k$, a one-parameter algebraic deformation of A is a family of algebras $\{A_t\}$ parameterized by $t$ such that $A_0\cong  A$ and the multiplicative structure of $A_t$ varies algebraically with $t$.  

M. Gerstenhaber considered the case where $A$ is an associative $k$-algebra and the deformation of $A$ is again associative.  

Let $A[[t]]$ be the $k[[t]]$-module of formal power-series with coefficients in the $k$-module $A.$ The algebra $A$ is a submodule of $A[[t]].$ Of course, $A[[t]]$ is an algebra where the algebra structure is obtained by bilinearly extending the multiplication of $A$, but one may also impose other multiplications on $A[[t]]$ that agree with that of $A$ when we specialize to $t=0$. Suppose a multiplication $m:A[[t]] \otimes _{k[[t]]}A[[t]] \rightarrow A[[t]] $ is given by a formal power-series of the form 
\begin{align*} 
m(a,b) = m_0(a,b) + m_1(a,b)t + m_2(a,b)t^2+ \cdots . 
\end{align*}  
Note that since the multiplication $m$ is defined over $k[[t]]$, it is enough to consider $a$ and $b$ in $A.$ It is further assumed that each $m_n$ is a linear map $A\otimes_k A\rightarrow  A$ with $m_0(a,b) = ab$ (multiplication in $A$). 

A {\em one-parameter formal deformation} of a $k$-algebra $A$ is  a formal power-series $m = \sum_{n=0}^\infty  m_nt^n$ with coefficients in $\mbox{Hom}_k(A\otimes_kA,A)$ such that  $m_0:A\otimes_k A\rightarrow A$ is the multiplication in $A$. The deformation is {\em associative} if $m(m(a,b),c) = m(a,m(b,c))$ for all $a,b,c$ in $A$.

M. Gerstenhaber introduced a notion of infinitesimal deformation and showed that under certain cohomological conditions it is possible to obtain deformed algebras. The associated cohomology (called deformation cohomology) is the Hochschild cohomology for associative algebras. He also introduced the notion of equivalence of deformed objects. 

The aim of the present article is to introduce and study the above problem in the equivariant world. More specifically, we consider the category {\bf As($G$)} of associative algebras equipped with actions of a finite group $G$ and equivariant associative algebra maps and study equivariant deformation theory in this category. Note that given any associative algebra $A$, any finite subgroup of the group of self isomorphisms of $A$ acts on $A.$ We denote an object of this category by $(G, A)$ and call it an action of $G$ on $A$ (or, simply, an action). Moreover, for a given action $(G, A),$ the action of $G$ on $A$ naturally extends to an action on $A[[t]].$ The algebra $A$ is a $G$-submodule of $A[[t]]$, and we could make $A[[t]]$ an algebra by bilinearly extending the multiplication of $A$ and the induced $G$-action is preserved in the sense that the induced multiplication is equivariant. We shall denote this by $(G, A[[t]]).$ The question in the equivariant context is to see whether there are other multiplications on $A[[t]]$ that agree with that of $A$ when we specialize to $t=0$ and preserve the group action. Moreover, equivariant deformation theory should classify deformations preserving symmetries of the deformed objects.

Let $\mbox{Hom}^G_k(A\otimes_kA,A)$ denote the space of all $k$-linear maps $A\otimes_k A \rightarrow A$ which are $G$-equivariant with the diagonal action of the group $G$ on $A \times A.$ 

Given an action $(G, A),$ ideally, a {\em one-parameter equivariant formal deformation} of $(G,A)$ should be a formal power-series $m = \sum_{n=0}^\infty  m_nt^n$ with coefficients 
$$m_i \in \mbox{Hom}^G_k(A\otimes_kA,A)$$ such that  $m_0:A\otimes_k A\rightarrow A$ is the multiplication in $A$. The deformation is {\em associative} if $m(m(a,b),c) = m(a,m(b,c))$ for all $a,b,c$ in $A$.

To this end,
\begin{enumerate}
\item we formulate equivariant Hochschild cohomology of $(G, A)$ with coefficients in a $G$-bimodule;
\item introduce the notion of invariant infinitesimal deformations;
\item study the problem of extending an infinitesimal deformation to a full-blown equivariant deformation;
\item and introduce the notion equivalence of deformed objects in the equivariant set up. 
\end{enumerate}      

Throughout the paper $k$ will be a fixed field and $G$ will denote a finite group. The tensor product over $k$ will be denoted by $\otimes$. A typical element of $A^{\otimes n}$ will be denoted by $(x_1, \ldots, x_n).$

The paper is organized as follows. In Section \ref{recall}, we recall known results and set up notations that we will use throughout the paper. In Section \ref{deformation-cohomology}, we define the notion of a $G$-bimodule over an associative algebra equipped with an action of $G$ and introduce equivariant Hochschild cohomology with coefficients in a $G$-bimodule. In Section \ref{equivariant-deformation}, we extend the classical formal one-parameter deformation theory of associative algebras in the equivariant context. We show that the equivariant Hochschild cohomology plays the role of the deformation cohomology in the sense that it controls equivariant deformation. Finally, in Section \ref {equivalence-of-deformed object}, we introduce the notions of equivalence of deformed objects and equivariantly rigid objects in the present context.

\section{Preliminaries}\label{recall}
In this section we gather standard information concerning deformations of associative algebras. The main purpose is to introduce the basic notation
and terminology that we will use throughout this paper. The standard references are \cite{G1, G2, L0}.

\begin{defn}\label{bimodule}
Let $A$ be an associative $k$-algebra. A bimodule over $A$ is a $k$-module $M$ equipped with a linear left $A$-action and a right $A$-action such that $(am)a^\prime = a(ma^\prime)$ for $a, a^\prime \in A$ and $m \in M.$ The actions of $A$ and $k$ on $M$ are assumed to be compatible, for instance:
$$(\lambda a)m = \lambda (am) = a(\lambda m),~~ \lambda \in k, ~a \in A,~ m \in M.$$ When $A$ has a unit element $1_A$, we always assume that $1_Am =m1_A = m$ for all $m \in M.$  
\end{defn}
We shall denote the product map of $A$ by $\mu : A\otimes A \rightarrow A. ~~ \mu (a, b) = ab.$ Note that $M = A$ is an $A$-bimodule where the actions are given by the product in $A.$

Let us briefly describe the definition of Hochschild cohomology of an an associative algebra $A$ with coefficients in a bimodule $M.$

Consider the module $C_n(A; M) := M\otimes A^{\otimes n}.$ The Hochschild boundary is the $k$-linear map $d : M\otimes A^{\otimes n} \rightarrow M\otimes A^{\otimes n-1}$ given by the formula

\begin{align*}
d(m,a_1, \ldots a_n):& = (ma_1, a_2, \ldots a_n)\\
& + \sum_{i=1}^{n-1}(-1)^i(m,a_1, \ldots, a_ia_{i+1}, \ldots, a_n) +(-1)^n(a_nm,a_1, \ldots a_{n-1}).
\end{align*}

This gives a chain complex $C_\sharp(A; M)=\{C_n(A; M), d\},$ where the module $M\otimes A^{\otimes n}$ is in degree $n,$ known as Hochschild complex. 

In the case where $M=A$ the Hochschild complex $C_\sharp(A;A),$ written simply as $C_\sharp(A)$  is 
$$C_\sharp(A) : \ldots \rightarrow A^{\otimes n+1}\stackrel{d}{\rightarrow}A^{\otimes n}\stackrel{d}{\rightarrow}\ldots \stackrel{d}{\rightarrow}A^{\otimes 2}\stackrel{d}{\rightarrow} A.$$    

Set $C^n(A;M) := Hom_k(C_{n-1}(A), M)$ and define Hochschild coboundary 
$$\delta : C^n(A;M) \rightarrow C^{n+1}(A;M)$$ as follows. For a cochain $f\in C^n(A; M)$, $\delta (f)$ is given by

\begin{align*}
& \delta (f)(x_1, \ldots , x_{n+1})\\ 
& = x_1f(x_2, \ldots , x_{n+1}) + \sum_{0< i <  n+1} (-1)^i f(x_1, \ldots,  x_ix_{i+1}, \ldots , x_{n+1}) \\
& + (-1)^{n+1}f(x_1, \ldots , x_n)x_{n+1}. 
\end{align*}

This gives a cochain complex $C^\sharp(A;M).$ The Hochschild cohomolgy groups are defined by 
$$H^n(A; M) := H_n(C^\sharp(A;M)).$$ 

For Gerstenhaber's theory we need Hochschild cohomology $H^3(A; A)$ of $A$ with coefficients in $A$, regarded as an $A$-bimodule. 
 
Let us briefly review the classical deformation theory of associative algebras. For deformation theory, we assume that the field $k$ is of characteristic zero.

Let A be an associative algebra over a field k of characteristic zero. Let $K=k[[t]]$ be the formal 
power-series ring with coefficients in $k$.  

\begin{defn} A formal one-parameter deformation of $A$ is a $K$-bilinear multiplication law $m_t : A[[t]]\times A[[t]] \longrightarrow A[[t]]$ on the space $A[[t]]$ of formal power-series in a variable $t$ with coefficients in $A$, satisfying the following properties:
$$ m_t(a, b) = m_0(a,b) + m_1(a, b)t + m_2(a, b)t^2 + \cdots ~~ \mbox{ for} ~ a, b \in A,$$ where $m_i :A\times A\longrightarrow A$ are $k$-bilinear and $m_0(a, b) = \mu (a,b) =ab$ is the original multiplication on $A$, and $m_t$ is associative.
\end{defn}
Thus, a formal deformation of $A$ is an associative algebra structure $m_t$ on $A[[t]]$ over $K$ such that for $t=0,$ we get back the given algebra structure on $A$. 

Note that the condition of associativity of $m_t$ is equivalent to the equation
$$m_t(m_t(a, b), c) = m_t(a, m_t(b, c)), ~~ \mbox{for} ~ a, b, c \in A.$$  
Writing down the above equation explicitly, we get

\begin{eqnarray}\label{equation-1}
\sum_ {\stackrel{p+q = r}{p, q \geq 0}} m_p (m_q (a,b),c) - m_p (a,m_q (b,c)) = 0,~r \geq 0.
\end{eqnarray}

\begin{remark}\label{case-zero-one}
For $r=0$, this is just the associativity of the original multiplication. For $r=1$, the above condition implies 
$$am_1(b,c) - m_1(ab,c) + m_1(a,bc) - m_1(a,b)c =0.$$  
In terms of Hochschild theory this simply means that $m_1$ is a $2$-cocycle, that is, $m_1 \in Z^2(A,A) = \{c\in C^2(A; A) : \delta (c) =0\}.$ 
The $2$-cocycle $m_1$ is called the infinitesimal of the deformation.
\end{remark}

Thus, we have
\begin{prop}
The infinitesimal of a one-parameter family of formal deformation of an associative algebra is a Hochschild $2$-cocycle.
\end{prop}

\begin{remark} The base of a formal one-parameter deformation as defined above is $K$. More generally, a one-parameter deformation of order $n$, is a deformation with base $K/(t^{n+1})$, given by $m_t$ modulo $(t^{n+1})$. In this case, the associativity condition means that the $2$-cochains $m_{r}$, satisfy (\ref{equation-1}) for $0\leq r\leq n$. 
\end{remark} 

Next, comes the question of extending an infinitesimal deformation to a full-blown deformation. If we start with an arbitrary $2$-cocycle $m_1 \in Z^2(A,A)$, it need not be
an `infinitesimal' of a deformation. If it be so, 
then we say that  $m_1$ is integrable. The integrability 
of $m_1$ implies an infinite sequence of relations 
which may be interpreted as the vanishing of 
the `obstructions' to the integration of $m_1$. 

For suppose we have a deformation of $A$ of order $n \geq 1$, given by $m_t$ modulo $(t^{n+1})$. Then by the above remark, $2$-cochains $m_{r}$, satisfy (1) for $0\leq r\leq n$. In order that the given $n^{th}$ order deformation extends to an $(n+1)^{th}$ order deformation $m_t$ modulo $(t^{n+2})$, over $K/(t^{n+2})$, the multiplication $m_t$ modulo $(t^{n+2})$ must be associative, equivalently, (\ref{equation-1}) should hold for $0\leq r\leq n+1$.  Using coboundary formula, the equation for $r = n+1$ can be rewritten as

\begin{eqnarray}
\delta m_{n+1}(a,b,c) = \sum_ {\stackrel {p+ q = n+1}{p, q > 0}} m_p (m_q(a,b),c) - m_p (a,m_q (b,c)).
\end{eqnarray}
In view of this observation, given an $n^{th}$ order deformation $m_t$ modulo $t^{n+1}$ over $K/(t^{n+1})$, we define a function $F$ by
$$F (a, b, c) = \sum_ {\stackrel {p + q=n+1}{p,~q > 0}} m_p (m_q(a,b),c) - m_p (a,m_q (b,c)),~a, b, c \in A.$$

Then, $F$ is a $3$-cochain and is called the $(n+1)^{th}$ obstruction cochain in extending  a deformation of order $n$ to a deformation of order $n+1$. One of the main observations in deformation theory is the following result.

\begin{thm}
The $(n+1)^{th}$ obstruction cochain $F$ is a $3$-cocycle and the given $n^{th}$ order deformation extends to a deformation of order $(n+1)$ if and only if the the cohomology class of $F$ vanishes.
\end{thm} 

We immediately obtain
\begin{corollary}
If $H^3(A,A) = 0$ then any Hochschild $2$-cocyle is integrable.
\end{corollary}

Given two associative deformations $(A[[t]], m_t)$ and $(A[[t]], n_t),$ a formal isomorphism
$$\Psi : (A[[t]], m_t)\rightarrow (A[[t]], n_t)$$  
is a k[[t]]-linear automorphism of the form 
$$\Psi (a) = \psi_0(a) + \psi_1(a)t + \psi_2(a)t^2 + \cdots$$
where  each $\psi_i$ is a $k$-linear map $A\rightarrow A,$ $\psi_0(a) = a$ for $a\in A$ and 
$$n_t(\Psi (a), \Psi (b)) = \Psi (m_t(a, b))$$ for all $a, ~b \in A.$ Observe that it is enough to consider $a \in A,$ since $\Psi$ is defined over $k[[t]].$	
In that case, the two deformations are said to be equivalent.

It is a standard fact that if $H^2(A; A) =0$ then, $A$ has only trivial deformation and in this case, $A$ is called rigid.

\section{Group actions and Equivariant Hochschild cohomology}\label{deformation-cohomology}
\begin{defn}\label{associative-algebra-group-action}
Let $A$ be an associative $k$-algebra with product $\mu (a, b) = ab$ and $G$ be a finite group. The group $G$ is said to act on $A$ from the left if there exists a function 
$$\phi : G\times A \rightarrow A,~~ (g, a) \mapsto \phi (g, a) = ga$$ satisfying the following conditions.
\begin{enumerate}
\item $ex= x$ for all $x \in A$, where $e \in G$ is the group identity.
\item $g_1(g_2x) = (g_1g_2)x$ for all $g_1, g_2 \in G$ and $x \in A$.
\item For every $g \in G$, the left translation $\phi_g = \phi (g, ~) : A \rightarrow A, ~~a \mapsto ga$ is a linear map.
\item For all $g\in G$ and $a, b \in A,$ $\mu (ga, gb) = g\mu (a, b)= g(ab),$ that is, $\mu$ is equivariant with respect to the diagonal action on $A\times A.$
\end{enumerate}
\end{defn}

We shall denote an action as above by $(G, A).$ 

The following is an equivalent formulation of the above definition.
\begin{prop}
Let $G$ be a finite group and $A$ be an associative algebra. Then $G$ acts on $A$ if and only if there exists a group homomorphism 
$$\psi : G \rightarrow \text{Iso}_{As} (A, A),~~g \mapsto \psi(g)=\phi_g$$ from the group $G$ to the group of algebra isomorphisms from $A$ to $A$.  
\end{prop}

\begin{proof}
Given an action $(G, A),$ define a map as follows.
$$\psi : G \rightarrow \text{Iso}_{As} (A, A),~~g \mapsto \psi(g)=\phi_g.$$
It is straightforward to check that it is a group homomorphism. Conversely, given a group homomorphism 
$$\psi : G \rightarrow \text{Iso}_{As} (A, A),$$ define a map
$G \times A \rightarrow A$ by $(g, a) \mapsto \psi (g)(a).$ It is easy to check that this is an action of $G$ on the associative algebra $A.$ 
\end{proof}

\begin{exam}
Let $G$ be a finite group and $V$ be a representation space of $G.$ Let 
$$T^0(V) =k, ~~T^r(V) = \overbrace{V\otimes \cdots \otimes V}^{r\mbox{-copies}}, ~~r>0,$$ 
Let $T(V) = \oplus_{r\geq 0}T^r(V)$ be the tensor algebra on $V.$ This is the free associative algebra on $V.$ The linear action on $V$ induces an action $T(V)$ to yield an action $(G, T(V)).$
\end{exam}

\begin{exam}
The symmetric group $S_n$  acts on  the algebra $M_n(\mathbb R)$ (the algebra of all $n\times n$ matrices over $\mathbb R$ with respect to matrix addition and matrix multiplication) by interchanging rows (or, columns).
\end{exam}

\begin{exam}
Let $X$ be a $G$-set and let $A = \{ \alpha: X\rightarrow \mathbb R\}$ be the vector space of all real valued functions on $X$. Observe that $A$ is an algebra with the product $\alpha \beta (x) = \alpha(x) \beta(x)$ for all $x \in X.$ Define an action of $G$ on $A$ by $(g, \alpha) \mapsto g\alpha,$ where $(g\alpha)(x) =\alpha (gx), ~~x\in X.$ Note that for $g\in G$ and $\alpha, \beta \in A,$ 
$$g(\alpha \beta) (x)= (\alpha \beta)(gx)= \alpha (gx) \beta (gx)= (g\alpha)(x) (g\beta)(x)= (g\alpha)(g\beta)(x)$$ for all $x\in X.$ Thus, $(G, A)$ is an action.  
\end{exam}

\begin{exam}\label{deformation-example}
Consider the algebra $A = k[x, y]/(y^2-x^3).$ The group $\mathbb Z_2$ acts on the algebra $k[x, y]$ by replacing $y$ by $-y$, keeping the variable $x$ unchanged. This action preserves the ideal $(y^2-x^3),$ hence, induces an action on the algebra $A.$ Similarly, we have an action $(\mathbb Z_2, A),$ where $A =  k[[x, y]]/(y^2-x^3).$
\end{exam}

Let $(G, A)$ be a given action and $H < G$ be a subgroup of $G.$ Then, the H-fixed point set $A^H,$ defined by
$$A^H = \{a \in A : ha = a ~~~\forall h \in H\},$$ is a subalgebra of $A$. We shall denote the multiplication $\mu|(A^H\times A^H)$ on $A^H$ by $\mu^H.$  Moreover, note that if $H, K$ are subgroups of $G$ with $g^{-1}Hg \subset K,~~g\in G,$ then the associative algebra homomorphism $\phi_g$ maps $A^K$ to $A^H$. 

\begin{defn}\label{system-fixedpoint-subalgebra}
The family $\{A^H: H<G\}$ is called the system of fixed point subalgebra.
\end{defn}

Next, we introduce a notion of $G$-bimodule over an associative algebra $A$ equipped with an action of $G$. 

\begin{defn}
A $G$-bimodule over $(G, A)$ is a bimodule $M$ over $A$ such that $G$ acts linearly on $M$ and the left $A$-action and the right $A$-action   
$$A\times M\rightarrow M,~~ M \times  A\rightarrow M$$ on $M$ are equivariant (cf. Definition \ref{bimodule}). 
\end{defn}

Observe that for an action $(G, A),$ $M=A$ considered as a $G$-module is a $G$-bimodule over $(G, A)$ in the above sense. 

We now introduce equivariant Hochschild cohomology of an action $(G, A)$ with coefficients in a $G$-bimodule $M.$

Set 
\begin{align*}
C^n_G(A; M) :& = \{c\in C^n(A; M): c(\phi_g (a_1), \ldots, \phi_g(a_n))= gc(a_1, \ldots, a_n),~~g\in G\}\\
& = \{ c\in C^n(A; M): c(g a_1, \ldots, ga_n) = gc(a_1, \ldots, a_n),~~g\in G\},
\end{align*}
where $C^n(A; M)$ is the $n^{th}$-cochain group of the algebra $A$ with coefficients in the bimodule $M.$ In other words, $C^n_G(A; M)$ consists of all Hochschild $n$-cochains which are equivariant. Clearly, $C^n_G(A; M)$ is a submodule of $C^n(A; M).$ An element $c\in C^n_G(A; M)$ will be referred to as an invariant $n$-cochain. 

\begin{lemma}
If a cochain $c$ is invariant then $\delta (c)$ is also invariant. In other words,
$$c \in C^n_G (A; M)\Longrightarrow \delta (c) \in C^{n+1}_G (A; M).$$
\end{lemma}

\begin{proof}
Let $c \in C^n_G (A; M)$ be an invariant $n$-cochain and $g \in G.$  

Then, for every $(x_1, \ldots, x_n) \in A^{\otimes n}$ we have
\begin{align}\label{equality-one}
& c (\phi_g(x_1), \ldots, \phi_g(x_n))= g(c(x_1, \ldots , x_n)).
\end{align}
Next, recall from the definition of Hochschild coboundary that
\begin{align*} 
& \delta (c)(\phi_g(x_1), \ldots , \phi_g(x_{n+1}))\\
& = gx_1c(gx_2, \ldots , gx_{n+1})\\
& + \sum_{1\leq i <  n+1} (-1)^i c(gx_1, \ldots,  gx_igx_{i+1}, \ldots , gx_{n+1})\\
& + (-1)^{n+1} c(gx_1, \ldots , gx_n)gx_{n+1}\\
& = g\delta (c)( x_1, \ldots, x_{n+1}).
\end{align*}
Therefore, 
$$\delta (c) \in C^{n+1}_G (A; M).$$
\end{proof}

Thus $C^\sharp_G(A; M) =\{C^n_G(A; M), \delta\}$ is a cochain subcomplex of the Hochschild cochain complex.

\begin{defn}\label{equivariant-complex}
The cochain complex $C^\sharp_G(A; M)$ will be referred to as the equivariant Hochschild cochain complex of $(G, A)$ with coefficients in the $G$-bimodule $M$.
\end{defn}

\begin{defn}\label{equivariant-deformation-complex}
We define the equivariant $n^{th}$-Hochschild cohomology group of $(G, A)$ with coefficients in $M$ by 
$$H^n_G(A; M):= H_n(C^\sharp_G(A; M)).$$ For $M =  A$ we will denote this group by $H^n_G(A; A).$
\end{defn}

\section{Equivariant deformation}\label{equivariant-deformation}
In this section, we will introduce formal one-parameter deformation theory of associative algebras equipped with finite group actions. For this section we assume that $k$ is a field of characteristic zero. As in section \ref{recall}, let $K =k[[t]]$ be the formal power-series ring.

Let $(G, A)$ be a given action of $G$ on an associative algebra $A$.
\begin{defn}\label{definition-equivariant-deformation}
An equivariant formal one-parameter deformation of $(G, A)$ is a $K$-bilinear multiplication law 
$$m_t : A[[t]]\times A[[t]] \longrightarrow A[[t]],$$ on the space $A^[[t]]$ of formal power-series in a variable $t$ with coefficients in $A,$ satisfying the following properties:
\begin{enumerate}
\item $ m_t(a, b) = m_0(a,b) + m_1(a, b)t + m_2(a, b)t^2 + \cdots ~~ \mbox{ for} ~ a, b \in A,$ where $m_i :A\times A\longrightarrow A$ are $k$-bilinear and $m_0(a, b) = \mu (a,b) = ab$ is the original multiplication on $A$, and $m_t$ is associative.

\item For every $g \in G,$ 
$$m_i (ga, gb) =  g m_i (a, b), ~~a, ~b \in A$$ for all $i \geq 1,$ that is $m_i \in \mbox{Hom}^G_k(A\otimes A, A)$ for all $i \geq 1.$ 
\end{enumerate}
\end{defn}
Note that the second condition of the above definition holds automatically for $i=0$ because of the fact that $\mu$ is equivariant.

As in the non-equivariant case, the condition that $m_t$ be associative is equivalent to the equation
$$m_t(m_t(a, b), c) = m_t(a, m_t(b, c)),$$ for all $a, b, c \in A.$   

Writing down the above equation explicitly, we get  

\begin{eqnarray}\label{associativity-explicit}
\sum_ {\stackrel{p + q =r}{p, q \geq 0}} m_p (m_q (a,b),c) - m_p (a,m_q (b,c)) = 0,~r \geq 0.
\end{eqnarray}

More generally, we have the notion of equivariant deformation of finite order.

\begin{defn}\label{equivariant-deformation-finite-order}
An equivariant formal one-parameter deformation of $(G, A)$ of order $n$ is a deformation of $A$ with base $K/(t^{n+1}),$ given by a multiplication law $m_t$ modulo the ideal $(t^{n+1})$ satisfying the following properties:
\begin{enumerate}
\item $ m_t(a, b) = m_0(a,b) + m_1(a, b)t + m_2(a, b)t^2 + \cdots ~~ \mbox{ for} ~ a, b \in A,$ where $m_i :A\times A\longrightarrow A$ are $k$-bilinear and $m_0(a, b) = \mu (a,b) = ab$ is the original multiplication $\mu$ on $A$, and $m_t$ is associative modulo $(t^{n+1}).$

\item $m_i \in \mbox{Hom}^G_k(A\otimes A, A)$ for all $i \geq 1.$ 
\end{enumerate}
\end{defn}

\begin{remark}\label{associativity-equivariant-finite-order}
Note that the associativity condition modulo $(t^{n+1})$ means that the $2$-cochain $m_r$, satisfy (\ref{associativity-explicit}) for $0\leq r\leq n.$ 
\end{remark}
Assume that we have an equivariant deformation $m_t$ of $A.$ Then, 
$$m_1 \in C^2_G (A; A)$$ by condition $(2)$ of Definition \ref {definition-equivariant-deformation}. Moreover, $\delta(m_1) = 0$ by the non-equivariant case (cf. Remark \ref{case-zero-one}).

Thus, the $2$-cochain $m_1 \in C^2_G (A; A)$
is a $2$-cocycle and hence, represents a cohomology class in $H_G^2(A; A).$ 

\begin{defn}\label{equivariant-infinitesimal}
We call this $2$-cocycle $m_1 \in C^2_G (A; A)$ as the infinitesimal of the given equivariant deformation. An equivariant deformation of order one is called an equivariant infinitesimal deformation.
\end{defn}

The following example provides an equivariant deformation of order one of the action given in Example \ref{deformation-example}.

\begin{exam}
Consider the action $(\mathbb Z_2, A),$ where $A = k[[x, y]]/(y^2-x^3),$ with $k$ a field of characteristic zero. We may identify $A$ with $k[[x]] +yk[[x]].$  

Let $A[[t]]$ be the module of formal power-series in the variable $t$ with coefficients in $A$. Clearly, $A[[t]]$ is a $\mathbb Z_2$-module. With the above identification of $A$ we may define a product 
$$m_t : A[[t]] \times A[[t]] \rightarrow A[[t]],~~m_t=m_0 + m_1t$$ as follows. Here, $m_0$ is the original multiplication in $A.$ For $p(x),~~q(x) \in k[[x]]$ define

$$m_t(p(x), q(x)) = p(x)q(x);~~ m_t(y, y) = x^3 + x^2t;$$
$$m_t(p(x), yq(x))= m_t(yq(x), p(x)) = yp(x)q(x);$$
$$ m_t (yp(x), yq(x)) = p(x)q(x)(x^3+x^2t).$$ Thus,  $m_1(yp(x), yq(x)) = p(x)q(x)x^2.$ A direct computation shows that $m_1$ satisfies Equation (\ref{associativity-explicit}) for $r=1.$ In other words, $m_1$ is a cocycle and hence, the above binary operation defines an associative algebra structure on $A[[t]].$ Moreover, it preserves the $\mathbb Z_2$-action on $A[[t]].$ For $t=0,$ we get back the original action $(\mathbb Z_2, A).$ Thus, $(A[[t]], m_t)$ is an equivariant infinitesimal deformation of $A.$
\end{exam} 

Next, we discuss the problem of extending an equivariant deformation of $A$ of order $n$ to that of $A$ of order $n+1.$ 

Suppose we have an equivariant deformation of $A$ of order $n \geq 1$, given by $m_t$ modulo $(t^{n+1}).$ Then by  Remark \ref{associativity-equivariant-finite-order}, the $2$-cochain $m_{r} \in C^2_G(A; A)$ satisfy (\ref{associativity-explicit}) for $0\leq r\leq n.$ In order that the given $n^{th}$ order deformation extends to an $(n+1)^{th}$ order deformation $m_t$ modulo $(t^{n+2}),$ that is, an equivariant deformation of $A$ over $K/(t^{n+2})$, the multiplications $m_t$ modulo $(t^{n+2})$ must be associative, equivalently, equations (\ref{associativity-explicit})  should hold for $0\leq r\leq n+1$.

We write down the equation for $r = n+1$ using the the definition of the coboundary $\delta$ as 
\begin{eqnarray}
\delta m_{n+1}(a,b,c) = \sum_ {\stackrel {p+ q = n+1}{p, q > 0}} m_p (m_q(a,b),c) - m_p (a,m_q (b,c)).
\end{eqnarray}

Define a function $F$ by
$$F (a, b, c) = \sum_ {\stackrel {p + q=n+1}{p,~q > 0}} m_p (m_q(a,b),c) - m_p (a,m_q (b,c)),~a, b, c \in A.$$

Then $F$ is a $3$-cochain, that is, $F \in C^3(A; A).$ 

\begin{lemma}
The $3$-cochain $F$ is invariant. In other words, $ F \in C^3_G(A; A).$
\end{lemma}
\begin{proof}
To prove that $F$ is invariant we must show that for every $g \in G$ the following holds:
$$F (ga, gb, gc) = gF(a, b, c),~~~\forall  a, b, c \in A.$$
Note that by Definition \ref{equivariant-deformation-finite-order}, we have  for all $a, b \in A,$
$$m_i(ga, gb) = gm_i(a, b).$$
Then, for all $a, b, c \in A,$
\begin{align*}
& F(ga, gb, gc)\\
& = \sum_ {\stackrel {p + q=n+1}{p,~q > 0}} m_p (m_q(ga,gb),gc) - m_p (ga,m_q (gb,gc))\\
& = \sum_ {\stackrel {p + q=n+1}{p,~q > 0}} m_p (gm_q(a,b),gc) - m_p (ga,gm_q (b,c))\\
& = \sum_ {\stackrel {p + q=n+1}{p,~q > 0}} gm_p (m_q(a,b),c) - gm_p (a, m_q (b,c))\\
& = g\sum_ {\stackrel {p + q=n+1}{p,~q > 0}} m_p (m_q(a,b),c) - m_p (a, m_q (b,c))\\
& = g F(a, b, c).
\end{align*}
Therefore, $F \in C^3_G(A; A).$
\end{proof}

\begin{defn}
The $3$-cochain $ F \in C^3_G(A; A)$ is called the $(n+1)^{th}$ obstruction cochain in extending the given equivariant deformation of order $n$ to an equivariant deformation of $A$ of order $n+1$.
\end{defn}

As in the non-equivariant case, we have the following result.
\begin{thm}
The $(n+1)^{th}$ obstruction cochain $ F$ is a $3$-cocycle and the given $n^{th}$ order equivariant deformation extends to an equivariant deformation of order $(n+1)$ if and only if the cohomology class of $F$ vanishes.
\end{thm} 
\begin{proof}
Note that  $F\in C_G^3(A;A) < C^3(A; A)$ is also the obstruction cochain for non-equivariant extension of the given deformation and 
$$\delta : C_G^n(A; A)\rightarrow C_G^{n+1}(A; A)$$ is the restriction of
$$\delta : C^n(A; A)\rightarrow C^{n+1}(A; A)$$ to the submodule $C_G^n(A; A).$ Therefore, $F$ is a cocycle in $C_G^3(A; A).$

If the cohomology class of $F$ vanishes then, there exists a $2$-cochain, say, $m_{n+1} \in C^2_G(A; A)$ such that
$$\delta m_{n+1}(a, b, c) = F(a, b, c) = \sum_ {\stackrel {p+ q = n+1}{p, q > 0}} m_p (m_q(a,b),c) - m_p (a,m_q (b,c)),$$ for all $a, b, c \in A.$ Then, we may use $m_{n+1}$ as the coefficients of $t^{n+1}$ to get an equivariant deformation of $A$ over $K/(t^{n+2}).$ Conversely, if the given deformation extends to an equivariant deformation of order $(n+1)$ then, (\ref{associativity-explicit}) holds for $0\leq r\leq n+1$ and in that case, as noted before, the obstruction cochain is a coboundary. 
Hence the result follows.
\end{proof}

\begin{corollary}\label{cohmological-condition-extension}
If $H^3_G(A; A) = 0$ then every $2$-cocycle in $C^2_G(A; A)$ of $(G, A)$ may be extended to an equivariant formal deformation.
\end{corollary}

\begin{remark}
Let $(G, A)$ be a given action. If $\{A[[t]], m_t\}$ is an equivariant formal one-parameter deformation of $(G, A)$, then the family 
$$\{A^H[[t]]=A[[t]]^H, m^H_t = m_t/(A^H[[t]]\times A^H[[t]]): H<G\}$$ is a family of formal one-parameter deformations for the system of fixed point subalgebras $\{A^H: H<G\},$ (cf. Definition \ref{system-fixedpoint-subalgebra}).
\end{remark}

\section{Equivalence of equivariant deformations}\label{equivalence-of-deformed object} 
The aim of this final section is to discuss briefly the notion of equivalence of equivariant deformations of $(G, A).$

First observe that for a given action $(G, A),$ the action of $G$ on $A$ naturally extends to an action on $A[[t]] = A\otimes K.$ The algebra $A$ is a $G$-submodule of $A[[t]]$, and we could make $A[[t]]$ an algebra by bilinearly extending the multiplication of $A$ and the induced $G$-action is preserved in the sense that the induced multiplication is equivariant. We shall denote this by $(G, A[[t]]).$
But we may also impose other multiplications on $A[[t]]$ that agree with that of $A$ when we specialize to $t=0$. 

\begin{defn}\label{formal-iso}
Given two associative equivariant deformations 
$$\{A[[t]], m_t\}~~~ \mbox{and}~~~ \{A[[t]], n_t \}$$ an equivariant formal isomorphism between them is a $k[[t]]$-linear $G$-automorphism  
$$\Psi : A[[t]]\rightarrow A[[t]]$$
of the form 
$$\Psi (a) = \psi_0(a) + \psi_1(a)t + \psi_2(a)t^2 + \cdots$$
where  each $\psi_i$ is an equivariant $k$-linear map $A\rightarrow A,$ $\psi_0(a) = a$ for $a\in A$ and 
$$n_t(\Psi (a), \Psi(b)) = \Psi(m_t(a, b))$$ for all $a, ~b \in A.$ 
\end{defn}
Note that as  $\Psi$ is defined over $k[[t]]$, it is enough to consider $a$ in $A.$ 

\begin{defn}
We say that the two equivariant deformations
$$\{A[[t]], m_t\}~~~ \mbox{and}~~~ \{A[[t]], n_t\}$$
are equivalent if there exists a formal isomorphism 
$$\Psi: \{A[[t]], m_t\}\rightarrow \{A[[t]], n_t\},$$
and we write 
$$\{A[[t]], m_t\} \cong \{A[[t]], n_t \}.$$
\end{defn}

For a given action $(G, A),$ if two equivariant deformations $\{A[[t]], m_t\}$ and  $\{A[[t]], n_t \}$ are equivalent with an equivariant formal isomorphism
$$\Psi: \{A[[t]], m_t\}\rightarrow \{A[[t]], n_t\},$$
then for every subgroup $H<G,$
$$\Psi^H: \{A^H[[t]], m^H_t\} \rightarrow \{A^H[[t]], n^H_t \}$$ is a formal isomorphism of non-equivariant deformations of $A^H.$

\begin{corollary}
For a given action $(G, A),$ if two equivariant deformations $\{A[[t]], m_t\}$ and  $\{A[[t]], n_t \}$ are equivalent then for every subgroup $H<G,$
$$\{A^H[[t]], m^H_t\} \cong \{A^H[[t]], n^H_t \}$$ as non-equivariant deformations of $A^H.$
\end{corollary}

Writing down the condition 
$$n_t(\Psi (a), \Psi(b)) = \Psi (m_t(a, b))$$ for all $a, ~b \in A$ (cf. Definition \ref{formal-iso}) and collecting the coefficients of  $t^n$ yields 
\begin{align}\label{explicit-product-preserv}
\sum_{i+j+k=n}  n_i(\psi_j(a),\psi_k(b)) = \sum_{i+j=n} \psi_i(m_j(a,b)).   
\end{align}

Observe that $\delta\psi_1 = m_1 - n_1.$ Thus, the $2$-cocycles $m_1$ and  $n_1$ are in the same cohomology class.

Given an equivariant $k$-linear map $\psi_1: A \rightarrow A$, we may ask when it may be extended to an equivariant formal isomorphism from $\{A[[t]], m_t\}$ to $\{A[[t]], n_t\}$.  For general $n$, equation (\ref{explicit-product-preserv}) may be written as
\begin{align}\label{extension-isomorphism}
\delta\psi_n(a,b) = \sum_{{i+j=n},\atop {i\neq n}} \psi_i(m_j(a,b)) - \sum_{{i+j+k=n},\atop {j,k\neq n}}  n_i(\psi_j(a),\psi_k(b)).
\end{align}
Given a truncated equivariant algebraic isomorphism  $\Psi = \sum \psi_it^i$,  $ i < n,$ define a $2$-cochain as follows.
For all $a,~b \in A,$ set
$$O_n(a, b) = \sum_{{i+j=n},\atop {i\neq n}} \psi_i(m_j(a,b)) - \sum_{{i+j+k=n},\atop {j,k\neq n}}  n_i(\psi_j(a),\psi_k(b)).$$

Then, $O_n\in C^2_G(A; A)$ is a $2$-cocyle. This $2$-cocyle is the obstruction to extending the given truncated isomorphism at the level $n.$ For suppose, the class represented by $O_n$ is zero. Then, we have an invariant cochain $\psi \in C^1_G(A; A)$ such that $\delta (\psi) = O_n.$ Now extend the given truncated isomorphism to the next level by taking $\psi_n = \psi$ as the coefficient of $t^n.$    

If all such obstructions vanish, then the two equivariant deformations are equivalent.

\begin{thm}
If $H^2_G(A; A) = 0,$ then all deformations of $(G, A)$ are isomorphic. 
\end{thm} 

\begin{defn}
An equivariant deformation of $(G, A)$ is called a trivial deformation if it is isomorphic to $(G, A[[t]]).$
\end{defn}

\begin{defn}
An action $(G, A)$ is called equivariantly rigid if it admits only trivial deformations.
\end{defn}

\begin{corollary}
If $H^2_G(A; A) = 0,$ then $(G, A)$ is rigid.
\end{corollary}

\end{document}